\documentclass[12pt]{article}
\usepackage{amsthm}

\usepackage[ansinew]{inputenc}
\usepackage{graphicx}
\usepackage{color}
\usepackage[colorlinks]{hyperref}

\usepackage{enumerate,latexsym}
\usepackage{latexsym}
\usepackage{amsmath,amssymb}
\usepackage{graphicx}
\usepackage{times}
\usepackage{float}
\usepackage[colorlinks]{hyperref}
\usepackage{hyperref}
\hypersetup{%
	colorlinks = true,
	linkcolor  = black
}
\newfont{\bb}{msbm10 at 11pt}
\newfont{\bbsmall}{msbm8 at 8pt}

\def\rth{\mathbb{R}^3}
\def\R{\mathbb{R}}
\def\B{\mathbb{B}}
\def\N{\mathbb{N}}

\def\S{\Sigma}
\def\C{\mathbb{C}}

\def\Pe{\mathbb{P}}

\def\M{\mathbb{M}}
\def\D{\mathbb{D}}
\def\Pe{\mathbb{P}}
\def\esf{\mathbb{S}}

\newcommand{\la}{\looparrowright}

\newcommand{\ben}{\begin{enumerate}}
\newcommand{\bit}{\begin{itemize}}
\newcommand{\een}{\end{enumerate}}
\newcommand{\eit}{\end{itemize}}
\newcommand{\wh}{\widehat}
\newcommand{\ds}{\displaystyle}
\newcommand{\Int}{\mbox{\rm Int}}

\newcommand{\Inj}{\mbox{\rm Inj}}

\newcommand{\wt}{\widetilde}

\newcommand{\ed}{\end{document}}
\newcommand{\ov}{\overline}

\newcommand{\sff}{\mathrm{I\!I}}

\def\a{{\alpha}}

\def\t{{\theta}}

\def\g{{\gamma}}
\def\G{{\Gamma}}
\def\l{{\lambda}}

\def\de{{\delta}}
\def\be{{\beta}}
\def\ve{{\varepsilon}}

\def\cB{\mathcal{B}}

\def\cM{\mathcal{M}}

\let\8=\infty \let\0=\emptyset

\def\cte.{\mathop{\rm cte.}\nolimits}

\def\Re{\mathop{\rm Re }\nolimits}

\def\N{\mathbb{N}}

\def\B{\mathbb{B}}

\def\R{\mathbb{R}}

\def\C{\mathbb{C}}

\def\D{\mathbb{D}}

\newtheorem{theorem}{Theorem}[section]
\newtheorem{lemma}[theorem]{Lemma}
\newtheorem{proposition}[theorem]{Proposition}

\newtheorem{remark}[theorem]{Remark}
\newtheorem{corollary}[theorem]{Corollary}
\newtheorem{definition}[theorem]{Definition}

\newtheorem{claim}[theorem]{Claim}

\numberwithin{equation}{section}

\usepackage{fullpage}

\usepackage{pdfsync}
\definecolor{pp}{rgb}{.5,0,.7}

\definecolor{rr}{rgb}{.8,0,.3}

\begin{document}

\begin{title}
{Geometry of branched minimal surfaces of finite index}
\end{title}

\begin{author}
{William H. Meeks III\thanks{This research was partially supported
 by CNPq - Brazil, grant no. 400966/2014-0. }
 \and Joaqu\'\i n P\' erez\thanks{Research of both
 authors was partially supported by
MINECO/MICINN/FEDER grant no. PID2020-117868GB-I00,
regional grant P18-FR-4049,
and by the “Maria de Maeztu” Excellence Unit IMAG,
reference CEX2020-001105-M, funded by MCINN/AEI/10.13039/501100011033/ CEX2020-001105-M.
}}
\end{author}
\maketitle
\vspace{-.6cm}

\begin{abstract}	
Given $I,B\in\N\cup \{0\}$, we investigate the
existence and geometry of complete finitely
branched minimal surfaces $M$ in $\rth$ with Morse index at most
$I$  and total branching order at most $B$. 
Previous works of Fischer-Colbrie~\cite{fi1} and Ros~\cite{ros9} explain that 
such surfaces are precisely the	complete minimal surfaces  in $\rth$ of finite total
curvature and finite total branching order. Among other things, 
we derive scale-invariant weak chord-arc type results for such an $M$ with 
estimates that are given in terms of $I$ and $B$. In order to obtain
some of our main results for these special surfaces, we obtain general intrinsic
monotonicity of area formulas for  $m$-dimensional submanifolds $\S$ of
an $n$-dimensional Riemannian manifold $X$, where these area estimates depend on the
geometry of $X$ and  upper bounds on the lengths of the mean curvature vectors of $\S$. 
We also describe a family of complete, finitely branched minimal surfaces in $\R^3$ 
that are stable and non-orientable; these examples generalize the classical Henneberg minimal surface.
\vspace{.3cm}


\end{abstract}
\maketitle

\section{Introduction}  \label{sec:introduction}
Let $X$ be a complete Riemannian $3$-manifold with positive
injectivity radius $\Inj(X)$.
Let $M$ be a complete immersed surface in $X$ of constant mean curvature
(CMC). The Jacobi operator of $M$ is the
Schr\"{o}dinger operator
\[
L=\Delta +|A_M|^2+ \mbox{Ric}(N),
\]
where
$\Delta $ is the Laplace-Beltrami operator on $M$, $|A_M|^2$ is the square
of the norm of its second fundamental form and $\mbox{Ric}(N)$ denotes
the Ricci curvature of $X$ in the direction of the unit normal vector $N$ to $M$;
the index of $M$ is the index of $L$,
\[
\mbox{Index}(M)=\lim _{R\to \infty }\mbox{Index}(B_M(p,R)),
\]
where $B_M(p,R)$ is the intrinsic metric ball in $M$ of
radius $R>0$ centered at a point $p\in M$, and
$\mbox{Index}(B_M(p,R))$ is the number of negative eigenvalues of $L$
on $B_M(p,R)$ with Dirichlet boundary conditions.
Here, we have assumed that the immersion
is two-sided (which is the case when the constant value $H$ of the mean curvature of 
$M$ is not zero).  In the case, the immersion is
one-sided, then the index is defined in a similar manner using compactly supported
variations in the normal bundle; see Definition~\ref{DefIndexNO} for details.

Given $I,B\in\N\cup \{0\}$, we investigate the
existence and geometry of complete finitely
branched minimal surfaces $M$ in $\rth$ with index at most
$I$  and total branching order at most $B$;
let $\cM(I,B)$ be the space of such examples.
Works of Fischer-Colbrie~\cite{fi1} and
Ros~\cite{ros9} ensure that the surfaces 
\[
M\in \bigcup_{I,B\in \N\cup \{ 0\}} \cM(I,B)
\]
are precisely the complete minimal surfaces  in $\rth$ of finite total curvature
and  finite total branching order.
One goal of this paper  is to derive certain scale-invariant weak chord-arc type 
results for surfaces in $\cM(I,B)$ with explicit estimates
given in terms of $I$ and $B$; see Proposition~\ref{propos5.5}
for these estimates. We also describe some interesting
new examples of non-orientable surfaces in $\cM(0,B)$, $B\geq 2$.
These new examples of complete stable branched minimal surfaces
generalize the classical Henneberg surface of finite total curvature $-2\pi$ 
that has two simple branch points; these surfaces are
described analytically and geometrically at the end of Section~\ref{sec3}.  
In Section~\ref{sec3}
we also explain how to extend to $\cM(I,B)$ the geometric and
topological lower bound estimates for the index of complete unbranched minimal surfaces
with finite total curvature due to Chodosh and Maximo~\cite{ChMa2}.

In Section~\ref{volgrowth} we study the area of intrinsic balls $B_M(x,R)$ of
an $n$-dimensional submanifold of a Riemannian $m$-manifold $M$, where
$x\in M$ and $0<R\leq \Inj(X)$.  In particular,
we derive explicit upper bounds for the area growth of
$B_M(x,R)$ as a function of $R\in (0,\Inj(X)]$, that depend on upper bounds 
for the sectional curvature of the extrinsic geodesic ball $B_X(x,R)$ and 
for the length of the mean curvature vector of $M$ restricted to  $B_M(x,R)$. 
In Section~\ref{section5} we will apply this intrinsic area estimate 
to obtain certain scale-invariant weak chord-arc bounds for any surface $M\in \cM(I,B)$;
see Proposition~\ref{propos5.5}.

The intrinsic area estimates in Section~\ref{volgrowth} will also be applied
in our papers~\cite{mpe19} and \cite{mpe18} to study CMC surfaces
of bounded index in spaces $X$ of dimension three.  The
monotonicity-of-area type formulae in Proposition~\ref{lemma8.4}, the
weak chord-arc results given in Proposition~\ref{propos5.5} and other theoretical 
results in Section~\ref{sec3}, such as the aforementioned extension of the Chodosh and Maximo
lower bound estimates for the index of surfaces in $\cM(I,B)$,
have important applications to the proof to the Hierarchy
Structure Theorem~1.1 in~\cite{mpe18}; this theorem is
a fundamental  result that describes the structure
of complete CMC surfaces of finite index in a 3-dimensional $X$
with $\Inj(X)>\de>0$ and having a fixed
an upper bound on its absolute sectional curvature function, and it was our
main motivation for developing the results in the present paper.
\vspace{.1cm}

{\em Acknowledgments}:
We thank Otis Chodosh for explaining to us his work~\cite{ChMa2}
with Davi Maximo on lower bounds for the
index of a complete branched minimal surface
in $\rth$ of finite total curvature, in terms of its genus and number of ends
counted with multiplicity, and how the analysis by Karpukhin~\cite{Kar1} can be
used to allow finitely many branch points.

\section{Volume growth of intrinsic balls in
	submanifolds of bounded mean curvature vector}
\label{volgrowth}
Let $M$  be an immersed $n$-dimensional submanifold in a geodesic ball $B_X(x_0,R_1)$ of an
$m$-dimensional manifold $(X,g)$, with $x_0\in M$ and $R_1$ less or than equal to the
injectivity radius function $\Inj_X(x_0)$ of $X$ at $x_0$. In this section we will find
lower bounds for the $n$-dimensional volume $A(r)$ of  $B_M(x_0,r)$, as a
function of $r\in (0,\Inj_X(x_0)]$;
see Proposition~\ref{lemma8.4} below for a precise description.

Let us denote by $\ov{\Delta}, \Delta $ the Laplacians in $X$ and $M$,
respectively. Analogously, $\ov{\nabla},\nabla$ will
stand for the Levi-Civita connections and
gradient operators. Let $N_{n+1},\ldots ,N_m$ be a local orthonormal basis of
the normal bundle to $M$, and let $\vec{H}$ be the mean curvature vector of $M$.
We start with a well-known formula.
\begin{lemma}
	\label{lemma8.1}
	Given $f\in C^{\infty }(X)$, $(\ov{\Delta}f)|_M
	=\Delta (f|_M)-n\, \vec{H}(f)+\sum_{j=n+1}^m g(\ov{\nabla}_{N_j}\ov{\nabla}f,N_j)$.
\end{lemma}
\begin{proof}
	Let $\{ v_1,\ldots ,v_n\} $ be a local orthonormal basis for  $TM$.
	\[
	\begin{array}{rcl}
		(\ov{\Delta}f)|_M&=& {\displaystyle \sum_{i=1}^{n}g(\ov{\nabla}_{v_i}\ov{\nabla}f,v_i)
			+\sum_{j=n+1}^m g(\ov{\nabla}_{N_j}\ov{\nabla}f,N_j)}
		\\
		&=& {\displaystyle
			\sum_{i=1}^{n}g(\ov{\nabla}_{v_i}(\nabla f+\sum_{j=n+1}^mN_j(f)N_j),v_i)
			+\sum_{j=n+1}^m g(\ov{\nabla}_{N_j}\ov{\nabla}f,N_j)}
		\\
		&=&{\displaystyle
			\sum_{i=1}^{n}g(\nabla_{v_i}\nabla f,v_i)
			+\sum_{j=n+1}^mN_j(f)\sum_{i=1}^{n}g(\ov{\nabla}_{v_i}N_j,v_i)
			+\sum_{j=n+1}^m g(\ov{\nabla}_{N_j}\ov{\nabla}f,N_j)}
		\\
		&=&{\displaystyle
			\Delta (f|_M)-n\, \vec{H}(f)+\sum_{j=n+1}^m g(\ov{\nabla}_{N_j}\ov{\nabla}f,N_j).}
	\end{array}
	\]
	\par
	\vspace{-1.2cm}
\end{proof}	

Given $a\in \R$, let $s_a(t)$ be the unique solution
of $x''(t)+a\, x(t)=0$, $x(0)=0$, $x'(0)=1$.
We will denote by $I_a$ the interval $[0,\pi/\sqrt{a})$
when $a>0$, and $I_a=[0,\infty)$ if $a\leq 0$.
Thus, $s_a(t)>0$ for all $t\in I_a\setminus \{ 0\} $.
Let $f_a\colon I_a\to \R$ be the smooth function given by
\begin{equation}
	\label{deff}
	f_a(t)=\frac{1}{t^2}\left(1-t\frac{s_{a}'(t)}{s_{a}(t)}\right) ,\ t\in I_a.
\end{equation}
A direct computation gives that
\begin{equation}\label{fa}
f_a(t)=\left\{ \begin{array}{cc}
	\frac{1}{t^2}\left(1-t\sqrt{a}\cot(\sqrt{a}t)\right) &  \mbox{if } a>0,
	\\
	0 & \mbox{if } a=0,
	\\
	\frac{1}{t^2}\left(1-t\sqrt{-a}\coth(\sqrt{-a}t)\right) &  \mbox{if } a<0.
\end{array}\right.
\end{equation}
The last equality implies that $f_a(t)$ is smooth at $t=0$, with value $f_a(0)=a/3$.

\begin{lemma}
	\label{lemma8.2}
	Let $R\colon B_X(x_0,R_1)\to [0,R_1)$ denote the extrinsic Riemannian distance function
	in $X$ to $x_0$.
	\begin{enumerate}
		\item The intrinsic Laplacian of the restriction of $R^2$ to $M$ is
		\[
		\Delta ((R^2)|_M)=2(m-1)R\, H^{S(R)}
		+2nR\, \vec{H}(R)+2|\nabla (R|_M)|^2-2R\sum_{j=n+1}^m\sff^{S(R)}(N_j^T,N_j^T),
		\]
		\newline
		where $H^{S(R)}$ denotes the mean curvature of
		the geodesic sphere $S(R)=\partial B_X(x_0,R)$
		with respect to the unit normal $-\ov{\nabla}R$,
		$N_j^T=N_j-N_j(R)\ov{\nabla}R$ is the projection of $N_j$ tangent to $S(R)$,
		and  $\sff^{S(R)}$ is the second fundamental form
		of $S(R)$ with respect to $-\ov{\nabla}R$.
		\item If the sectional curvature of $X$ satisfies
		$K_{\mbox{\rm \footnotesize sec}}\leq a$ for some $a\in \R$, then
		\begin{equation} \label{eq:LR2formula}
			\Delta ((R^2)|_M)\geq 2n+2nR\, \vec{H}(R) -2R^2f_a(R)\left( n-|\nabla (R|_M)|^2\right) ,
		\end{equation}
		and equality holds in (\ref{eq:LR2formula})
		if $K_{\mbox{\rm \footnotesize sec}}=a$. In particular if $X$ is flat, then
		\begin{equation} \label{eq:LR2formula3}
			\Delta ((R^2)|_M)=2n+2nR\, \vec{H}(R).
		\end{equation}
	\end{enumerate}
\end{lemma}
\begin{remark}
	{\rm
		For $a=0$, equation (\ref{eq:LR2formula3}) generalizes the well-known formula
		$\Delta ((R^2)|_M)=2n$ for minimal submanifolds
		of Euclidean space. Similarly, if we assume
		$K_{\mbox{\rm \footnotesize sec}}\leq 0$, inequality (\ref{eq:LR2formula})
		generalizes the inequality $\Delta ((R^2)|_M)\geq 2n$
		for minimal submanifolds given by Yau in~\cite[equation (7.1)]{yau6}.
	}
\end{remark}
\begin{proof}
	Lemma~\ref{lemma8.1} applied to $R^2$ gives
	\begin{equation}
		\label{ap30a}
		\Delta ((R^2)|_M)=(\ov{\Delta}(R^2))|_M+2nR\,
		\vec{H}(R)-\sum_{j=n+1}^mg(\ov{\nabla}_{N_j}\ov{\nabla}(R^2),N_j).
	\end{equation}
	We now compute the first and third terms of the last
	RHS. On the one hand, since $|\ov{\nabla}R|=1$,
	\begin{equation}
		\label{ap30b}
		\ov{\Delta}(R^2)=2+2R\,\ov{\Delta}R.
	\end{equation}
	As $\ov{\nabla}R$ is unitary and orthogonal to the geodesic
	spheres centered at $x_0$, we can take an orthonormal basis
	of $TX$ of the form $\{ E_1,\ldots,E_{m-1},\ov{\nabla}R \}$
	where $E_1,\ldots ,E_{m-1}$ is an orthonormal basis of the tangent
	space to $S(R)$. Thus,
	\[
	\ov{\Delta}R=\sum_{i=1}^{m-1}g(\ov{\nabla}_{E_i}\ov{\nabla}R,E_i)
	+g(\ov{\nabla}_{\ov{\nabla}R}\ov{\nabla}R,\ov{\nabla}R).
	\]
	
	The first term in the last RHS equals $(m-1)H^{S(R)}$,
	and the second term clearly vanishes. Thus,
	\begin{equation}
		\label{ap30c}
		\ov{\Delta}R=(m-1)H^{S(R)},
	\end{equation}
	and
	\begin{equation}
		\label{ap30ca}
		\ov{\Delta}(R^2)\stackrel{(\ref{ap30b})}{=}2+2(m-1)R\, H^{S(R)}.
	\end{equation}
	On the other hand,
	\begin{equation}
		\label{ap30cb}
		g(\ov{\nabla}_{N_j}\ov{\nabla}(R^2),N_j)=2g(\ov{\nabla}_{N_j}(R\ov{\nabla}R),N_j)
		=2N_j(R)^2+2R\, g(\ov{\nabla}_{N_j}\ov{\nabla}R,N_j).
	\end{equation}
	Decomposing $N_j=N_j^T+N_j(R)\ov{\nabla}R$ where $N_j^T$ is tangent to $S(R)$,
	the bilinearity of the second term of the last RHS with respect to $N_j$ allows us to write
	\begin{equation}
		\label{ap30cc}
		g(\ov{\nabla}_{N_j}\ov{\nabla}R,N_j)=g(\ov{\nabla}_{N_j^T}\ov{\nabla}R,N_j^T)
		=\sff^{S(R)}(N_j^T,N_j^T),
	\end{equation}
	where we have used that $g(\ov{\nabla}_{\ov{\nabla}R}\ov{\nabla}R,\ov{\nabla}R)=0$ and that
	$g(\ov{\nabla}_{N_j^T}\ov{\nabla}R,\ov{\nabla}R)=0$
	because $\ov{\nabla}R$ has constant length.
	From (\ref{ap30a}), (\ref{ap30ca}), (\ref{ap30cb})
	and (\ref{ap30cc}) we have
	\begin{equation}
		\label{ap30c'}
		\Delta ((R^2)|_M)=2+2(m-1)R\, H^{S(R)}+2nR\, \vec{H}(R)-2\sum_{j=n+1}^mN_j(R)^2
		-2R\sum_{j=n+1}^m\sff^{S(R)}(N_j^T,N_j^T).
	\end{equation}
	Since $\ov{\nabla}R=\nabla (R|_M)+\sum_jN_j(R)N_j$, then
	\begin{equation}
		\label{ap30c''}
		1=|\ov{\nabla}R|^2=|\nabla (R|_M)|^2+\sum_{j=n+1}^mN_j(R)^2.
	\end{equation}
	Plugging (\ref{ap30c''}) into (\ref{ap30c'}) we obtain item~1 of the lemma.
	
	As for item~2, we will assume that $K_{\mbox{\rm \footnotesize sec}}\leq a$
	for some $a\in \R$.
	Let $e_1,\ldots ,e_{m-1}$ be an orthonormal basis of principal directions
	of $T_xS(R)$, with respective principal curvatures
	$\l_1,\ldots ,\l_{m-1}$ with respect to the unit
	normal $-\ov{\nabla}R$ to  $S(R)$. For each $j=n+1,\ldots ,m$
	we can write $N_j^T=\sum_{i=1}^{m-1}a_{ij}e_i$ where $a_{ij}=g(e_i,N_j^T)=g(e_i,N_j)\in \R$.
	Thus,
	\[
	(m-1)H^{S(R)}=\sum_{i=1}^{m-1}\l_i\qquad \mbox{and}\qquad
	\sff^{S(R)}(N_j^T,N_j^T)= \sum_{i=1}^{m-1}\l_ia_{ij}^2.
	\]
	Hence, we can write the formula in item~1 of the lemma as
	\begin{equation}
		\label{ap30cd}
		\Delta ((R^2)|_M)=2R\sum_{i=1}^{m-1}\l_i\left( 1-\sum_{j=n+1}^ma_{ij}^2\right)
		+2nR\, \vec{H}(R)+2|\nabla (R|_M)|^2.
	\end{equation}
	
	Observe that given any tangent vector $v$ to $S(R)$,
	\begin{equation}
		\label{ap30ga}
		\sff^{S(R)}(v,v)=g(\ov{\nabla}_v\ov{\nabla}R,v)=(\ov{\nabla }^2R)(v,v),
	\end{equation}
	where $\ov{\nabla }^2R$ denotes the hessian of $R$.
	Since $K_{\mbox{\rm \footnotesize sec}}\leq a$,
	standard comparison results (see e.g.~\cite[Theorem 27]{pet1}) give
	\begin{equation}
		\label{ap30gb}
		\frac{s_{a}'(R)}{s_{a}(R)}\, g_R\leq \ov{\nabla}^2R,
	\end{equation}
	where $g_R$ is the induced metric by $g$ on $S(R)$.
	Evaluating (\ref{ap30gb}) at the principal directions $e_i$, we have
	\begin{equation}
		\label{ap30h}
		\frac{s_{a}'(R)}{s_{a}(R)} \leq \l_i,\quad
		\mbox{ for all }i=1,\ldots ,m-1.
	\end{equation}
	
	Given $i=1,\ldots ,m-1$, we decompose $e_i$ in its tangent and normal components to $M$ as
	\[
	e_i=e_i^{T,M}+\sum_{j=n+1}^mg(e_i,N_j)N_j=e_i^{T,M}+\sum_{j=n+1}^ma_{ij}N_j,
	\]
	from where
	\[
	1=|e_i|^2\geq \left| \sum_{j=n+1}^ma_{ij}N_j\right| ^2=\sum_{j=n+1}^ma_{ij}^2.
	\]
	This last inequality together with (\ref{ap30cd}) and (\ref{ap30h}), give
	\[
	\begin{array}{rcl}
		\Delta ((R^2)|_M)&\geq &{\displaystyle
			2R\frac{s_{a}'(R)}{s_{a}(R)}\sum_{i=1}^{m-1}\left( 1-\sum_{j=n+1}^ma_{ij}^2\right)
			+2nR\, \vec{H}(R)+2|\nabla (R|_M)|^2}
		\\
		&=&{\displaystyle
			2R\frac{s_{a}'(R)}{s_{a}(R)}\left( m-1-\sum_{j=n+1}^m|N_j^T|^2\right)
			+2nR\, \vec{H}(R)+2|\nabla (R|_M)|^2}
		\\
		&\stackrel{(\ref{deff})}{=}&{\displaystyle
			2\left( 1-R^2f_a(R)\right) \left( m-1-\sum_{j=n+1}^m|N_j^T|^2\right)
			+2nR\, \vec{H}(R)+2|\nabla (R|_M)|^2}
		\\
		&=&{\displaystyle
			2(m-1)-2\sum_{j=n+1}^m|N_j^T|^2-2R^2f_a(R)\left( m-1-\sum_{j=n+1}^m|N_j^T|^2\right)}
		\\
		& & {\displaystyle +2nR\, \vec{H}(R)+2|\nabla (R|_M)|^2}
		\\
		&\stackrel{(*)}{=}&{\displaystyle
			2n-2R^2f_a(R)\left( n-|\nabla (R|_M)|^2\right) +2nR\, \vec{H}(R),}
	\end{array}
	\]
	where in $(*)$ we have used that
	\[
	1-|\nabla (R|_M)|^2+\sum_{j=n+1}^m|N_j^T|^2
	\stackrel{(\ref{ap30c''})}{=}\sum_{j=n+1}^mN_j(R)^2
	+\sum_{j=n+1}^m|N_j^T|^2=\sum_{j=n+1}^m|N_j|^2=m-n.
	\]
	Now inequality (\ref{eq:LR2formula}) is proved.
	If  $K_{\mbox{\rm \footnotesize sec}}=a$, then
	both (\ref{ap30gb}) and (\ref{ap30h}) are equalities,
	and the above argument shows that (\ref{eq:LR2formula}) is also an equality.
	In the case $X$ is flat, then $a=0$ and
	$f_{0}(t)=0$, which gives (\ref{eq:LR2formula3}).
\end{proof}

The next result generalizes the classical monotonicity of area formula
of Allard~\cite[Section~5.1]{al1} for hypersurfaces of bounded mean curvature,
in part since it does not require the hypersurface to be proper in the
ambient space. Proposition~\ref{lemma8.4} is motivated by
the calculations in the last two pages of Yau~\cite{yau6}, where he derived the
lower bound area estimate given in \eqref{eq:lemma8.4} when $a\leq 0$, $H_0=0$.

\begin{proposition}[Intrinsic monotonicity of area formula]
	\label{lemma8.4}
	Let $\ov{B}_X(x_0,R_1)$ denote a closed geodesic ball in an
	$m$-dimensional manifold $(X,g)$, where $0<R_1\leq \Inj_X(x_0)$,
	and suppose that $K_{\mbox{\rm \footnotesize sec}}\leq a$ on $B_X(x_0,R_1)$
	for some $a\in \R$.
	Given $H_0\geq 0$, define
	\begin{equation} \label{eq:R0}
		R_0(a,H_0)=\left\{ \begin{array}{cl}
			\frac{1}{\sqrt{a}}\mbox{\rm arc cot}\left( \frac{H_0}{\sqrt{a}}\right) &  \mbox{if } a>0,
			\\
			1/H_0 & \mbox{if } a=0 \quad \mbox{(if $H_0=0$ we take $R_0(0,0)=\infty$)}
			\\
			\frac{1}{\sqrt{-a}}\mbox{\rm arc coth}\left( \frac{H_0}{\sqrt{-a}}\right), &  \mbox{if } a<0
			\quad \mbox{(if $\frac{H_0}{\sqrt{-a}}\geq 1$ we take $R_0(a,H_0)=\infty$),}
		\end{array}\right.
	\end{equation}
	and let $r_1=r_1(R_1,a,H_0)=\min \{ R_1,R_0(a,H_0)\}$.
	
	Suppose   $M$ is a complete, immersed, connected  $n$-dimensional submanifold of $X$
	and $x_0\in M$ is a point such that when $\partial M\neq \varnothing$,
	$d_M(x_0,\partial M)\geq R_1$ and the length of the mean curvature vector $\vec{H}$ of $M$
	restricted to $\ov{B}_X(x_0,R_1)$ is bounded from above by  $H_0$.
	Then:
	\ben
	\item If $M$ is compact without boundary, then  there exists $y\in M$
	such that the extrinsic distance from $x_0$ to $y$ is greater than or equal to $r_1$.
	\item The $n$-dimensional volume $A(r)$ of  $B_M(x_0,r)$ is
	a strictly increasing function of $r\in (0,r_1]$.
	\item For all $r\in (0,r_1]$ when $r_1\neq \infty$ or  otherwise, for all $r\in (0,\infty)$:
\begin{equation}
\label{eq:lemma8.4}
A(r)\geq \left\{ \begin{array}{cc}
\omega_n\, r^n e^{-nH_0r} & \mbox{ if $a\leq 0$},
\\
\omega_n\, r^n e^{-nr(H_0+\frac{1}{2}f_a(r_1)r)} & \mbox{ if $a>0$},
		\end{array}
		\right.
	\end{equation}
	\een
	where $\omega _n$ is the volume of the unit ball in $\R^n$ and the function $f_a$ is defined in~\eqref{fa}.
\end{proposition}
\begin{proof}
Let $\M^m(a)$ denote the $m$-dimensional, simply-connected space form
of constant sectional curvature $a\in \R$. Recall that the number $R_0(a,H_0)$ represents
the radius of a geodesic sphere in $\M^m(a)$ with constant mean curvature $H_0$, and
that geodesic spheres in $\M^m(a)$ of radii less than $R_0(a,H_0)$ have mean curvature greater than $H_0$.
	
We first prove item~1 of the lemma.
Fix a point $x_0\in M$ and let $r\in (0,r_1)$.
Suppose $M$ is compact with empty boundary and suppose
that  there does not exist a point $y\in M$ such that
the extrinsic distance from $x_0$ to $y$ is greater than $r_1$; 
in this case we have $M\subset \ov{B}_X(x_0,r_1)$.
Since $r_1\leq R_1\leq \Inj_X(x_0)$, all the distance spheres
$\partial B_X(x_0,r)$ with $r\in (0,r_1)$ are geodesic spheres.
Since the absolute sectional curvature of $X$
is bounded by $a$, comparison results imply that
$\partial B_X(x_0,r)$ has normal curvatures greater than $H_0$ because $r<R_0(a,H_0)$
in this case.
	Assume for the moment that $M$ is contained in $\ov{B}_X(x_0,r)$. As $M$
	is closed, there exists a largest $r_2\in (0,r]$ such that
	$M\subset \overline{B}_X(x_0,r_2)$, and by compactness of $M$ there exists
	a point  $x\in
	M\cap \partial B_X(x_0,r_2)$. Therefore, all the normal
	curvatures of $M$ at $x$ are greater than $H_0$, which implies that
	the length of the mean curvature vector of $M$ is greater than $H_0$,
	thereby contradicting one of the hypotheses
	on $M$.
	This contradiction proves that $M(x_0)$ cannot be contained in
	$\overline{B}_X(x_0,r)$. Since this non-containment equation holds for every $r\in (0,r_1)$
	and $M$ is compact, we conclude that $M$ cannot be contained in
	$B_X(x_0,r_1)$. 
	Item~1 is now proved.
	
To see that item~2 holds, consider two  values $r_2<r_3$ in $[0,r_1]$.
By item~1 and the hypotheses on $M$, then $B_M(x_0,r_3)\setminus \overline{B}_M(x_0,r_3)$
is a non-empty open subset of $M$, hence $A(r_2)<A(r_3)$.

	It remains to prove the lower bound estimates for $A(r) $ given in item~3.
	In what follows, we only consider values $r\in (0,r_1]$.
	By Stokes' Theorem,
	\begin{equation}
		\label{ap30f}
		\int_{B_M(x_0,r)}\Delta ((R^2)|_M)\leq \int_{\partial B_M(x_0,r)}
		|\nabla(R^2)|
		=2\int_{\partial B_M(x_0,r)}R|\nabla R|
		\leq 2r\, l(r),
	\end{equation}
	where $l(r)=\mbox{Volume}(\partial B_M(x_0,r))=A'(r)$ is the
	$(n-1)$-dimensional volume of $\partial B_M(x_0,r)$.
	
	Since $ K_{\mbox{\rm \footnotesize sec}}\leq a$,
	inequality~\eqref{eq:LR2formula}
	implies that
	\begin{equation}
		\label{ap30fa}
		\int_{B_M(x_0,r)}\Delta ((R^2)|_M)\geq 2n\, A(r)
		+2n \int_{B_M(x_0,r)}R\, \vec{H}(R)
		-2\int_{B_M(x_0,r)}R^2f_a(R)\left( n-|\nabla (R|_M)|^2\right) .
	\end{equation}
	Since $R\leq r$ and $|\vec{H}(R)|=|\langle \vec{H},\ov{\nabla}R\rangle |
	\leq |\vec{H}|\leq H_0$, we have $R\, \vec{H}(R)\geq -H_0r$, and thus,
	\begin{equation}
		\label{ap30fb}
		\int_{B_M(x_0,r)}R\, \vec{H}(R)\geq -H_0r\, A(r).
	\end{equation}
	
	Next we analyze the second integral in the RHS of (\ref{ap30fa}).
	
	If $a=0$, then $f_a\equiv 0$ and the second integral in the RHS of
	(\ref{ap30fa}) vanishes. In this case, \eqref{ap30f}, \eqref{ap30fa}
	and~\eqref{ap30fb} give
	\begin{equation}\label{ap30f'}
		2r\, A'(r)=2r\, l(r)\geq \int_{B_M(x_0,r)}\Delta ((R^2)|_M)\geq
		2n\, A(r)
		-2n H_0r\, A(r),
	\end{equation}
	that is,
	\[
	\frac{A'(r)}{A(r)}\geq \frac{n}{r}-nH_0 \qquad \forall r\in (0,r_1],
	\]
	which implies that
	\[
	\frac{d}{dr}\left( \frac{A(r)}{r^n e^{-nH_0r}}\right) \geq 0,
	\]
	hence the function
	\[
	r\mapsto \frac{A(r)}{r^n e^{-nH_0r}}
	\]
	is non-decreasing for $r\in (0,r_1]$.
	Since the limit as $r\to 0^+$ of this last function
	is $\omega_n$, we deduce:
	\[
	\mbox{If $a=0$, then } A(r)\geq \omega_nr^n e^{-nH_0r} \quad \mbox{for all }r\in (0,r_1].
	\]
	In fact, the last estimate holds if $a\leq 0$, because $f_a\leq 0$ in
	$(0,\infty)$ in this case, and hence,
	\begin{equation}\label{4.22}
		-2\int_{B_M(x_0,r)}R^2f_a(R)\left( n-|\nabla (R|_M)|^2\right) \geq 0
	\end{equation}
	so the same computations of case $a=0$ are valid for $a\leq 0$.

We next study the case $a>0$. Now $f_a(t)$ is strictly positive,
increasing in the interval $I_a=[0,\pi/\sqrt{a})$ and limits
to $a/3>0$ as $t\to 0^+$ and to $+\infty$ as $t\to (\pi/\sqrt{a})^-$.
Whenever $r\in (0,r_1]$,
	\begin{eqnarray}
		2r\, A'(r)&\stackrel{\eqref{ap30f}}{\geq }&
		\int_{B_M(x_0,r)}\Delta ((R^2)|_M)
		\nonumber
		\\
		&\stackrel{\eqref{ap30fa},\eqref{ap30fb}}{\geq }&
		2n\, A(r)-2n H_0r\, A(r)-2 \int_{B_M(x_0,r)}
		R^2f_a(R)\left( n-|\nabla (R|_M)|^2\right) \nonumber
		\\
		&\stackrel{(A)}{\geq }&
		2n\left[ 1-r^2f_a(r_1)\right] A(r)-2n rH_0\, A(r)\label{4.23}
	\end{eqnarray}
where in $(A)$ we have bounded $R\leq r$,
$f_a(R)\leq f_a(r_1)$ and $ n-|\nabla (R|_M)|^2\leq n$.
	
Finally, \eqref{4.23} implies
\[
\frac{d}{dr}\left( \frac{A(r)}{r^n e^{-nr(H_0+\frac{1}{2}f_a(r_1)r)}}\right) \geq 0;
\]
hence the function
\[
r\mapsto \frac{A(r)}{r^n e^{-nr(H_0+\frac{1}{2}f_a(r_1)r)}}
\]
is non-decreasing for $r\in  [0,r_1]$.
Since the limit as $r\to 0^+$ of this last function
is $\omega_n$, we deduce the
inequality (\ref{eq:lemma8.4}) in the case where $a>0$; thus, the proposition is proved.
\end{proof}

\begin{remark}
\label{remark8.5}
{\em 
\begin{enumerate}
\item Proposition~\ref{lemma8.4} holds regardless whether or not the normal bundle of
the submanifold $M$ is trivial, since item~2 of Lemma~\ref{lemma8.2}
does not depend on whether or not the normal bundle of $M$ admits a global
trivialization.
\item The proof of Proposition~\ref{lemma8.4} shows that if $M$ has local density $k\in \N$ at $x_0$, 
then the RHS in~\eqref{eq:lemma8.4} cab replaced by $k$ times the same expression.
\item In the case $a>0$, it holds that $A(r)\geq \omega_n\, r^n e^{-nr(H_0+\frac{1}{2}f_a(r)r)}$
for every $r\in (0,r_1]$. This can be proved by following the same proof for values $r\leq r_1'$ where
$r_1'$ is any number less than or equal to $r_1$.
\item If $H_0\neq 0$ or $a\neq 0$, the inequality~\eqref{eq:lemma8.4} is strict.
\end{enumerate}
	
	}
\end{remark}

\begin{corollary}
	\label{corol11.6}
	Let $R_1>0$, $a\in \R$ and $H_0\geq 0$,	and suppose  that $X$
	is a complete Riemannian $m$-dimensional manifold with
	injectivity radius at least $R_1>0$ and  $K_{sec}\leq a$.
	If $M\la X$ is a complete, non-compact immersed
	$n$-dimensional submanifold with empty boundary and the mean curvature
	vector $\vec{H}$ of $M$ satisfies $|\vec{H}|\leq H_0$, then $M$ has infinite
	volume.
\end{corollary}
\begin{proof}
Let $r_1>0$ be the number given by Proposition~\ref{lemma8.4}. 	
Observe that by Proposition~\ref{lemma8.4}, the $n$-dimensional volume of each component of $M$ 
is at least $A(r_1)>0$
Therefore, if $M$ has infinitely many components, then $M$ has infinite volume. 
So assume that $M$ has a finite number of components. Since $M$ is non-compact, then we can replace $M$ 
by a non-compact component. Take a point $x_0\in M$ and let $\g\colon [0,\infty)\to M$ a 
length-minimizing ray starting at $x_0$ and parameterized by arc length. Consider the pairwise disjoint intrinsic balls
$B_M(\g(2kr_1),r_1)$, $k\in \N$. Since each of these balls has volume at least $A(r_1)$
by Proposition~\ref{lemma8.4}, then we conclude that $M$ has infinite volume.
\end{proof}

\begin{proposition}
	\label{yau}
	Given $R_1>0$, $a\in \R$ and $H_0\geq 0$, there exists
	$r_2=r_2(R_1,a,H_0)\in (0,r_1]$
	(here $r_1$ is given by Proposition~\ref{lemma8.4})
	such that if $X$ is a complete Riemannian 3-manifold with
	injectivity radius at least $R_1>0$ and
	$K_{sec}\leq a$, and if $M\la X$ is
	a complete, connected immersed surface with boundary,
	whose mean curvature vector $\vec{H}$ satisfies
	$|\vec{H}|\leq H_0$, then for all $p\in \Int(M)$ we have
	\begin{equation} \label{yaulemma0}
		\mathrm{Area}[B_M(p,r)]\geq 3 r^{2}, \quad
		\mbox{whenever $0<r\leq \min \{r_2,d_M(p,\partial M)\} $.}
	\end{equation}
	
	Furthermore, given $\ve_0>0$ define $\ds C_A
	=\min \{ \ve_0,\frac{r_2^2}{\ve_0}\} $. If $p\in M$ satisfies
	$d_M(p,\partial M)\geq \ve_0$, then
	\begin{equation}
		\label{yaulemma2}
		\mathrm{Area}[B_M(p,d_M(p,\partial M))]\geq C_A\, d_M(p,\partial M)
	\end{equation}
	and
	\begin{equation}
		\label{yaulemma1}
		\mathrm{Area}[B_M(p,\ve_0)]\geq C_A\, \ve_0,
	\end{equation}
\end{proposition}
\begin{proof}
First suppose that $a>0$. By~\eqref{eq:lemma8.4}, we have that
whenever $0<r\leq \min \{r_1,d_M(p,\partial M)\} $,
\begin{equation}
\label{12.28}
\mathrm{Area}[B_M(p,r)]\geq \pi \, r^2 e^{-2r(H_0+\frac{1}{2}f_a(r_1)r)}
=\phi(r)r^2,	
\end{equation}
where $\phi(r)=\pi  e^{-2r(H_0+\frac{1}{2}f_a(r_1)r)}$
for all $r>0$. Choose $r_2=r_2(R_1,a,H_0)\in (0,r_1]$ such that
$\phi(r_2)\geq 3$, which can be done since $\phi$ is
continuous and $\phi(0)=\pi$. As $r>0\mapsto \phi(r)$ is decreasing, we have that if
$0<r\leq (0,\min \{r_2,d_M(p,\partial M)\} $, then
\[
\mathrm{Area}[B_M(p,r)]\stackrel{\eqref{12.28}}{\geq}
\phi(r)r^2\geq \phi(r_2)r^2\geq 3r^2,
\]
which proves~\eqref{yaulemma0} assuming $a>0$. The proof of~\eqref{yaulemma0} when $a\leq 0$
is similar and we leave it for the reader.
	
	Next assume that $p\in M$ satisfies $d_M(p,\partial M)\geq \ve_0$,
	and we will show that \eqref{yaulemma2} and \eqref{yaulemma1} hold.
	Let $\g\colon [0,d_M(p,\partial M))\to M$ be a minimizing geodesic
	from $p$ to $\partial M$, parameterized by
	arc length.
	Choose the largest $k\in \N$ such that
	\begin{equation}
		(2k-1)\ve_0\leq  d_M(p,\partial M)<(2k+1)\ve_0\leq3k\ve_0.
		\label{12.28a}
	\end{equation}
	By the triangle inequality, the collection  $\cB=\{B_M(\g(2(i-1)\ve_0),\ve_0)\}_{i=1}^{k}$
	is pairwise disjoint and $\cup \cB$ is contained in $B_M(p,d_M(p,\partial M))$; hence,
	\begin{equation}
		\mathrm{Area}[B_M(p,d_M(p,\partial M))]\geq
		\sum_{i=1}^{k}\mbox{\rm Area}\left( B_M(\g(2(i-1)\ve_0),\ve_0)\right) .
		\label{12.28c}
	\end{equation}
	
	Also observe that given $i\in \{ 1,\ldots ,k\}$, \eqref{12.28a} implies
	\begin{equation}
		\ve_0\leq d_M(\g(2(i-1)\ve_0),\partial M).
		\label{12.28d}
	\end{equation}
	We next prove~\eqref{yaulemma2}
	and~\eqref{yaulemma1} by consideration of two
	cases.
	\begin{itemize}
		\item \underline{Suppose $\ve_0\leq r_2$.}
		By \eqref{12.28d}, for each $i\in \{ 1,\ldots ,k\}$ we have
		\[
		\ve_0\leq \min \{r_2,d_M(\g(2(i-1)\ve_0),\partial M)\} .
		\]
		The last inequality allows us to use~\eqref{yaulemma0} to conclude that
		\begin{equation}
			\mathrm{Area}[B_M(\g(2(i-1)\ve_0),\ve_0)]\geq 3 \ve_0^{2}.	\label{12.28b}
		\end{equation}
		Note that $C_A=\ve_0$ in this case. Taking $i=1$
		in~\eqref{12.28b}, we have
		$\mathrm{Area}[B_M(p,\ve_0)]\geq 3\ve_0^{2}>\ve_0^2=C_A\ve_0$,
		hence \eqref{yaulemma1} holds.
		As the collection $\cB$ is pairwise disjoint,
		\eqref{12.28c} and~\eqref{12.28b} imply
		\[
		\mathrm{Area}[B_M(p,d_M(p,\partial M))]\geq  3k\ve_0^2
		=3kC_A\ve_0\stackrel{\eqref{12.28a}}{\geq} C_A\, d_M(p,\partial M),
		\]
		hence~\eqref{yaulemma2} also holds in this case.
		\item \underline{Suppose $\ve_0>r_2$.}
		By \eqref{12.28d}, for each $i\in \{ 1,\ldots ,k\}$ we have
		$r_2<d_M(\g(2(i-1)\ve_0),\partial M)$;
		hence~\eqref{yaulemma0} implies that
		\begin{equation}
			\mathrm{Area}[B_M(\g(2(i-1)\ve_0),\ve_0)]\geq 3r_2^2.	
			\label{12.28e}
		\end{equation}
		Since $ d_M(p,\partial M)<3k\ve_0$ and $C_A= \frac{r_2^2}{\ve_0}$ in this case,
		\[
		\mathrm{Area}[B_M(p,d_M(p,\partial M))]\geq 3k r_2^2
		=3kC_A \ve_0\stackrel{\eqref{12.28a}}{>} C_A d_M(p,\partial M),
		\]
		which proves that \eqref{yaulemma2}. The inequality
		\eqref{yaulemma1} follows from~\eqref{yaulemma2}
		after replacing $M$ by the closure of $B_M(p,\ve_0)$.
	\end{itemize}
	\par
	\vspace{-0.7cm}
\end{proof}

\begin{remark} \label{remark8.8}
	{\em
		A straightforward adaptation of the proof of Proposition~\ref{yau}
		gives a related statement and proof for any $n$-dimensional
		submanifold $M$, with a fixed upper bound on the length of its mean
		curvature vector field, in a Riemannian $m$-manifold $X$
		which has injectivity radius at least $R_1>0$ and
		sectional curvature bounded from above by some $a\in \R$; in this setting,
		$3r^2$ in~\eqref{yaulemma0} is
		replaced by $c_nr^n$, where $c_n$ is any positive number less
		than $\omega_n$.
	} 
\end{remark}

\section{Index of finitely branched minimal surfaces in $\R^3$}
\label{sec3}
\begin{definition}
{\rm
Let $\Sigma$ be a smooth surface endowed with a conformal class of metrics.
We say that a harmonic map $f\colon \Sigma\to \R^3$
is a (possibly non-orientable) branched minimal surface if it
is a conformal immersion outside of a locally finite
set of points $\cB_{\Sigma}\subset \Sigma$, where $f$ fails
to be an immersion.
Points in $\cB_{\Sigma}$ are called branch points of $f$.
It is well-known (see e.g. Micallef and
White~\cite[Theorem~1.4]{miwh1}) that given $p\in
\cB_{\Sigma}$, there exist a conformal coordinate $(\ov{\D},z)$
for $\Sigma$ centered at $p$ (here $\ov{\D}$ is the closed unit
disk in the
plane), a diffeomorphism $u$ of $\ov{\D}$ and a rotation
$\phi$ of $\R^3$ such that $\phi \circ f\circ u$ has the form
\[
z\mapsto (z^q,x(z))\in \C\times \R\sim \R^3
\]
for $z$ near $0$, where $q\in \N$, $q\geq 2$, $x$
is of class $C^2$, and $x(z)=o(|z|^q)$.
The branching order $B(p)\in \N$
is defined to be $q-1$. The total branching
order of $f$ is
\[
B(\Sigma):=\sum_{p\in \cB_{\Sigma}}B(p).
\]
}
\end{definition}

\begin{definition} \label{DefIndexNO}
	{\rm
		Given a $1$-sided minimal immersion $F\colon M\la X$,
		let $\wt{M}\to M$ be the two-sided cover of
		$M$ and let $\tau \colon \wt{M}\to \wt{M}$ be the associated deck
		transformation of order 2. Denote by $\wt{\Delta}$,
		$|\wt{A}|^2$ the Laplacian and squared norm of the second fundamental
		form of $\wt{M}$, and let $N\colon \wt{M}\to TX$
		be a unitary normal vector  field.  The index of $F$ is defined as the number
		of negative eigenvalues of the elliptic, self-adjoint operator
		$\wt{\Delta} +|\wt{A}|^2+\mbox{Ric}(N,N)$ defined over the space of
		compactly supported smooth functions $\phi \colon \wt{M}\to \R$
		such that $\phi \circ \tau =-\phi $.
	}
\end{definition}

We next recall a fundamental lower bound for the index
$I(f)$ of a connected, complete, possibly finitely branched
minimal surface $f\colon \S \la \R^3$ with  finite total
curvature, which is due to Chodosh and Maximo~\cite{ChMa2}, and to Karpukhin~\cite{Kar1}:
\begin{equation}
\label{eq:CMindex1}
3I(f)\geq \left\{ \begin{array}{ll}
{\displaystyle 2g(\S)+2\sum_{j=1}^{e}(d_j+1)-2B-5}
& \mbox{if $\S$ is orientable,}
\\
{\displaystyle g(\wt{\S})+2\sum_{j=1}^{e}(d_j+1)-2B-4}
& \mbox{if $\S$ is non-orientable,}
\end{array}\right.
\end{equation}
where $g(\Sigma)$ is the genus of $\S$ if $\S $ is
orientable (resp. $g(\wt{\S})$ is the genus of
the orientable cover $\wt{\S}$ of $\S$ if $\S $ is not
orientable\footnote{If $\Sigma$ is a compact non-orientable surface and	
$\wh{\S}\stackrel{2:1}{\to}\S$ denotes the
oriented cover of $\S$, then the genus of $\wh{\S}$ plus 1 equals the number 
of cross-caps in $\S$.}), $e$ and $B$ are respectively
the number of ends and the total branching order of $\Sigma$,
and for each end $E_j$ of $\S$, $d_j$ is the multiplicity
of $E_j$ as a multi-graph over the limiting tangent plane
of $E_j$.

Inequality \eqref{eq:CMindex1} has not been explicitly stated in the literature,
so an explanation is in order. Ros~\cite{ros9} proved that $3I(f)\geq 2g(\S)$
using harmonic square integrable $1$-forms on $\S$ for a minimal immersion
$f\colon \S \la \R^3$ with  finite total curvature, in order to produce test
functions for the index operator of $f$. Chodosh and Maximo~\cite[Theorem 1]{ChMa2}
improved Ros' technique with an enlarged space of harmonic $1$-forms which
admit certain singularities at the ends of $\S$ that take care of the spinning (multiplicity)
of each end of such an immersion $f$, obtaining a simplified version of~\eqref{eq:CMindex1}
without the term $-2B$. Finally, Karpukhin~\cite[Proposition 2.3 and Remark 2.4]{Kar1}
included the study of branch points although he made use
of the original space of $L^2(\S)$
harmonic $1$-forms considered by Ros. Formula~\eqref{eq:CMindex1} is the combined inequality
that one can deduce from~\cite{ChMa2} and~\cite{Kar1}.

The class of complete, non-flat, finitely branched,
stable minimal surfaces in $\R^3$ contains an interesting non-trivial
family of surfaces, as we explain next.
\begin{enumerate}
\item Any non-orientable, complete, finitely branched minimal
surface $f\colon \Sigma \la \R^3$ with finite total
curvature, whose extended unoriented Gauss map $G\colon \Pe^2
\to \Pe^2$ is a diffeomorphism, is stable (observe that the
conformal compactification of $\S$ must be $\Pe^2$). We prove
this property by contradiction: if $f$ is not stable, then
the first eigenvalue $\l_1$ of the Jacobi operator on
$\Sigma$ is negative, which implies that there exists an
eigenfunction $\phi \colon \esf^2\to \R$ of the lifted Jacobi
operator on the orientable cover $\pi \colon
\wt{\S}\to \S$ of $\Sigma$ so that $\phi \circ \tau =-\phi $ and
$L\phi +\l \phi =0$ on $\wt{\S}$, where $\l <0$ and $\tau \colon
\esf^2\to \esf^2$ is the antipodal map. Let $\Omega $
be a component of $\phi^{-1}(0,\infty)$. As $\phi $ is odd,
$\tau (\Omega )\subset \phi^{-1}(-\infty ,0)$ and
so, $\pi|_{\Omega }\colon \Omega \to \pi (\Omega)$ is a
diffeomorphism. In particular, $\pi(\Omega )$ is an
orientable domain in $\S$. Since $G$ is also a diffeomorphism,
$G(\pi (\Omega))$ is an orientable domain in $\Pe^2$. Thus,
$G(\pi (\Omega))$ lifts to two disjoint diffeomorphic domains
in $\esf^2$ of the form $g(\Omega )$, $(g\circ \tau)
(\Omega )$ (here $g\colon \wt{\S}\to \esf^2$ is the
Gauss map of $\widetilde{\S}$).
In particular, Area$((g\circ \tau)(\Omega ))=\mbox{ Area}(g(\Omega ))
\leq 2\pi$, which implies that the first eigenvalue of the
Jacobi operator $L$ on $\Omega $ is non-negative. This is
a contradiction, as the first Dirichlet eigenvalue of $L$ on
$\Omega $ (defined as the supremum of the first Dirichlet
eigenvalues of $L$ on a increasing sequence of compact smooth
domains $\Omega _i\nearrow \Omega $) is $\l<0$. This
contradiction proves that $\Sigma $ is stable.
			
\item Using the Weierstrass representation for non-orientable
minimal surfaces in~\cite{me7}, the classical
Henneberg minimal surface given by the
Weierstrass data\footnote{This means that $f(z)=
\Re\left(\int^z(\frac{1}{2}(1-g^2)\omega,\frac{i}{2}(1+g^2)
\omega ,g\omega  )\right)$ parameterizes the surface.}
on its oriented covering $\C\setminus \{ 0\}$
\[
g(z)=z,\quad \omega=z^{-4}(z^4-1)\, dz,
\]
 is a non-orientable, complete branched minimal surface
$f\colon\mathbb{P}^2\setminus\{ 0,\infty\}\la \R^3$ with two branch
points of order 1 at $\{ 1,-1\} , \{ i,-i\} \in \Pe^2$ and a
single end of spinning 3 at $\{ 0,\infty \}$.
Since  its extended
Gauss map  is a diffeomorphism from $\mathbb{P}^2$ to
$\mathbb{P}^2$, then the  Henneberg minimal surface $H_1=f(\mathbb{P}^2\setminus \{ 0,\infty\})$
is stable. 
After translating the surface in $\R^3$ so that
$f(e^{i\pi/4})=\vec{0}$, the branch points are mapped by $f$
into $\pm(0,0,1)$.

Henneberg's surface can be generalized
as follows. Given an odd integer $m\in \N$, consider the following
Weierstrass data on $\C\setminus \{0\}$,
\[
g(z)=z,\quad \omega=z^{-(3+m)}(z^{2m+2}-1)\, dz
\]
which produces a two-sheeted
cover of a complete minimal Mobius strip $f\colon
\mathbb{P}^2\setminus\{ 0,\infty\}\la \rth$ which is stable with $m+1$ branch
points of order 1 at the $(m+1)$ pairs of antipodal $(2m+2)$-roots of unity
and a single end of spinning $m+2$ at $\{ 0,\infty\}$.
Henneberg's minimal surface corresponds to the case $m=1$.
After translating the surface $H_m=f(\mathbb{P}^2\setminus\{ 0,\infty\})$ in $\R^3$
so that $f(e^{i\frac{\pi}{2(m+1)}})=\vec{0}$, the branch points of
$H_m$ are located at $(0,0,\pm
\frac{2}{m+1})$, and a parameterization of $H_m$ in polar
coordinates is
\[
f(re^{i\t})=\left( \begin{array}{c}x_1\\x_2\\x_3\end{array}
\right)
=\left(
\begin{array}{c}
\frac{1}{2} \left(\frac{r^{m} \cos (m\theta)}{m}+\frac{\cos
((m+2)\theta  )}{(m+2) r^{m+2}}\right)
-\frac{1}{2} \left(\frac{r^{m+2} \cos ((m+2)\theta  )}
{m+2}+\frac{\cos (m\theta)}{m r^{m}}\right)
\\
\frac{1}{2} \left(\frac{\sin ((m+2)\theta  )}{(m+2) r^{m+2}}
-\frac{r^{m} \sin (m\theta)}{m}\right)-\frac{1}{2}
\left(\frac{r^{m+2} \sin ((m+2)\theta  )}{m+2}
-\frac{\sin (m\theta )}{m r^{m}}\right)
\\
\frac{1}{m+1}\left(r^{m+1}+\frac{1}{r^{m+1}}\right)
\cos ((m+1)\theta )
\end{array}
\right) .
\]
We note that $f$ maps each of the $m+1$ pairs of opposite half-lines
\[
\left\{ l_j=\left\{ re^{i\frac{\pi j}{2(m+1)}}\ | \ r>0\right\},-l_j\right\}
\]
(for each $j$ odd) into a horizontal line of $\R^3$ that passes through
$\vec{0}$, and the union $L$ of these $m+1$ horizontal lines forms an equiangular
system contained in $H_m$. Therefore, the
reflection in $\C\setminus \{ 0\}$ about $l_j\cup (-l_j)$
induces a symmetry of $H_m$. Reflections in the $m+1$ vertical planes that
bisect each of the angles between the lines in $L$ are planes of symmetry of
$H_m$. Rotations of angle $\pi$ about each of the lines in $L$ together with
these $m+1$ planar reflections form the group of isometries of $H_m$ (all of
which extend to ambient isometries),
which, when considered to be a subgroup of $O(3)$, 
is the antiprismatic group $A_{2(m+1)}$. In fact, every intrinsic isometry of
$H_m$ extends to an extrinsic isometry, since such an intrinsic isometry produces a
conformal diffeomorphism of $\C\setminus \{0\}$ into itself that preserves the set of
$(2m+2)$-roots of unity.

For odd $m\geq 3$, these generalized Henneberg surfaces $H_m$
can be deformed to less symmetric examples
of non-orientable, complete finitely branched stable
minimal surfaces in $\R^3$ whose branch locus consists of
$m+1$ pairs of antipodal points
in $\C\setminus \{ 0,\infty \}$ ($H_1$ can be proven to be
the unique such surface for $m=1$); see~\cite{mope1} for a description
and special properties of these
deformed Henneberg-type examples.
\end{enumerate}

\section{Scale invariant weak chord-arc type estimates
for branched minimal surfaces of finite index in $\R^3$}
\label{section5}
\begin{proposition}
\label{propos5.5}
Given $I,B\in \N\cup \{ 0\} $, let $f\colon(\S, p_0)\la (\R^3,\vec{0})$
be a complete, connected, pointed branched
minimal surface with index at most $I$ and total branching order at most $B$.
Given $R>0$, let $\Omega _R$ denote  the component
of $f^{-1}(\ov{\B}(R))$ that contains $p_0$. Then, the following scale-invariant
estimates hold and depend only on $I$, $B$:
\ben
\item For any $p\in\Omega_R$,
\begin{equation}
\label{eq:lemma5.50}
d_{\Omega_R}(p,\partial \Omega _R)<\wh{L}R,
\end{equation}
where $\wh{L}=\sqrt{\frac{1}{2}(3I+2B+3)}$.
\item If $f $ is  injective with image  a plane, then the
distance between any two points of $\Omega_R$ is less than
or equal to $2R$. Otherwise, given points $p,q$ in
$\Omega_R$,
\begin{equation}
\label{eq:lemma5.53}
d_{\Omega_{2R}}(p,q)< \wh{C}R,
\end{equation}
where $\wh{C}=\wh{C}(I,B)=8\wh{L}^3+2\pi
\wh{L}^2-20\wh{L} -\frac{\pi}{2}$. In particular,
$\Omega_{R}\subset B_{\S}(p,\wh{C}R)$ for every $p\in \Omega_R$.
\een
\end{proposition}

\begin{proof}
Since (\ref{eq:lemma5.50}) 
and \eqref{eq:lemma5.53} are invariant under re-scaling,
we do not lose generality by assuming
$R=1$. Let $f\colon(\S, p_0)\la (\R^3,\vec{0})$
be a complete, connected, pointed branched
minimal surface  in $\R^3$ with index $I(f)\leq I$
and total branching order $B(\S)\leq B$.
Observe that such an $f$ has finite total curvature~\cite{fi1,fs1,ros9}.
Thus, $f$ is proper and $\Omega_1 $
is compact with non-empty boundary $\partial \Omega_1$.
Given a point $p\in \Int(\Omega_1)$, let $L=d_{\Omega_1}(p,\partial \Omega_1)$ and
consider a length minimizing  geodesic  arc
parameterized by arc length $\g\colon [0,L]\to \S$
joining $\g(0)=p$ to $\partial \Omega_1$. Observe
that the intrinsic ball of center $p$ and radius $L$ satisfies
$\ov{B}_\Sigma (p,L)\subset \Omega_1$.
The intrinsic version of the monotonicity
formula for minimal surfaces described in
Proposition~\ref{lemma8.4} applied to the particular case $m=3$,
$a=0$, $R_1=\infty$, $n=2$ and $H_0=0$,
gives that $\mbox{Area}[B_{\S}(p,L)]\geq \pi L^2$
(with the notation of Proposition~\ref{lemma8.4}, the number
$r_1=r_1(R_1,a,H_0)$ equals $\infty $ in this case; observe
that the proof of Proposition~\ref{lemma8.4} works for
branched minimal surfaces; also see the last page of
Yau~\cite{yau6} for the special case $a\leq 0$, $H_0=0$ in inequality~\eqref{eq:lemma8.4}).
Hence,
\begin{equation}
\label{eq:l55A}
\mbox{Area}(\Omega_1)\geq \mbox{Area}[B_{\S}(p,L)]\geq \pi L^2.
\end{equation}
	
Next we deduce an upper bound for $\mbox{Area}(\Omega_1)$.
Inequality (\ref{eq:CMindex1}) implies that
regardless of the orientability character of $\S$, we have
$3I\geq 3 I(f)\geq 2S+2e-2B(\Sigma)-5$,
where $e$ is the number of ends of $\S$, and
$S$ is the total spinning of the ends. Hence,
\begin{equation}
\label{eq:lemma5.5A}
2S\leq 3I-2e+2B(\Sigma)+5\leq 3I-2e+2B+5.
\end{equation}
As $e\geq 1$, we have
\begin{equation}
\label{eq:l55B}
\mbox{Area}(\Omega_1)\leq
\mbox{Area}[f^{-1}(\ov{\B}(1))]\stackrel{(\star)}{\leq }
\pi S\stackrel{(\ref{eq:lemma5.5A})}{\leq }\frac{\pi }{2}\left( 3I+2B+3\right) =\pi \wh{L}^2,
\end{equation}
where in $(\star)$ we have used that the asymptotic
area growth of $\S$ in balls of large radius $R$ is $\pi
SR^2$ (see e.g., \cite{jm1}) and the classical (extrinsic)
monotonicity formula. Now, (\ref{eq:l55A}) and
(\ref{eq:l55B}) give
\begin{equation}
\label{eq:l55C}
d_{\Omega_1}(p,\partial \Omega_1)=L\leq \wh{L}
\end{equation}
for any $p\in \Int(\Omega_1)$, which implies
that $d_{\Omega_R}(p,\partial \Omega _R)\leq\wh{L}R$;
notice that this last inequality is
strict (otherwise $f(\Sigma)$ is a possibly branched plane passing through
the origin by the extrinsic monotonicity formula,
in which case the first inequality in~\eqref{eq:lemma5.5A}
is strict). This implies that the inequality~\eqref{eq:lemma5.50} is strict,
and item 1 of Proposition~\ref{propos5.5} is proven.

In order to obtain item~2, we will need the following
auxiliary property: If  $f$ is not an embedded plane, then
\begin{equation}
\label{eq:lemma5.51}
\textstyle{d_{\Omega_R}(p,q)\leq 2\wh{L}(3I+2B-1)R
+\frac{1}{2}\mbox{\rm Length}(\partial \Omega_R).}
\end{equation}

Observe that (\ref{eq:lemma5.51}) is invariant under re-scaling.
We will divide the proof of~(\ref{eq:lemma5.51}) into four
claims.
\begin{claim}\label{claim4.2}
For any $p,q\in \Omega_1$,
\begin{equation}
\label{transverse}
d_{\Omega_1}(p,q)\leq \sup_{p',q'\in \Omega_1} d_{\Omega_1}
(p',q')	=\lim_{r\searrow 1} \;\sup_{p'',q''\in\Omega_{r}}
d_{\Omega_r} (p'',q'').
\end{equation}
\end{claim}
\begin{proof}
The first inequality in~\eqref{transverse} holds by
definition of supremum and so Claim~\ref{claim4.2} reduces to checking
that the equality part of~\eqref{transverse} holds. For each  $r\in (1,2]$, let
$p_r,q_r$ be points of $\Omega_r$ such that
\[
L_r:=d_{\Omega_r}(p_r,q_r)=\sup_{p'',q''\in\Omega_{r}}
d_{\Omega_r} (p'',q'')
\]
and let $\a_r\colon[0,L_r]\to
\Omega_r\subset \Omega_2$ be a Lipschitz curve contained in
$\Omega_r$ with Lipschitz constant 1 that realizes the
minimum distance $L_r$ in $\Omega_r$ between $p_r,q_r$.
Taking a sequence $r_j\searrow 1$, after passing to a
subsequence we obtain a limit Lipschitz curve $\a_1$ of the
$\a_{r_j}$ with Lipschitz constant~1 joining points
$p_1,q_1\in \Omega_1$ of positive length
$L_1:=\lim_{r_j\searrow 1}L_{r_j}$. It straightforward
to check  that $L_1=d_{\Omega_1}(p_1,q_1)=
\sup_{p',q'\in \Omega_1} d_{\Omega_1} (p',q')$,
which shows the equality part on the RHS of
\eqref{transverse}.
\end{proof}
\begin{claim}
\label{claim4.3}
If~\eqref{eq:lemma5.51} holds whenever
$\Omega_R$ is transverse to $\esf^2(R)$ along its boundary,
then~\eqref{eq:lemma5.51} holds for $R=1$ (and thus, it also
holds for any $R>0$).
\end{claim}
\begin{proof}
This is a direct consequence of Claim~\ref{claim4.2} since
almost all spheres centered at $\vec{0}$ are transverse to
$f$ by Sard's theorem.
\end{proof}

By Claim~\ref{claim4.3}, we can reduce the
proof of~\eqref{eq:lemma5.51} to the case that $R=1$ and $f$ is transverse to $\esf^2(1)$
along $\partial \Omega_1$. This transversality assumption implies that $\Omega_1$ is a
smooth, connected, compact surface with a finite set
$\{\partial_1,\ldots, \partial_b\}$ of boundary components, $b\in \N$.

\begin{claim}
\label{claim4.4}
For any $p,q\in\Omega_1$, $d_{\Omega_1}(p,q)\leq
2b\wh{L} +\frac{1}{2}\mbox{\rm Length}(\partial \Omega_1)$.
\end{claim}
\begin{proof}
Assuming $b>1$, there is a geodesic	arc $\a_1\subset
\Omega_1$ that minimizes the distance from $\partial _1$ to
the set $\cup_{i=2}^b \partial _i$,	and, possibly after
re-indexing, we may assume that $\a_1$ joins $\partial_1$ to
$\partial_2$. Notice that the distance from the midpoint
of $\a_1$ to $\partial \Omega_1$ is half the length of
$\a_1$, and so \eqref{eq:lemma5.50} implies that the length of
$\a_1 $ is less than $2\wh{L}$.  Assuming that $b>2$, let
$\a_2$ be a minimizing geodesic in $\Omega_1$ from
$\partial_1\cup \partial_2$ to the set  $\cup_{i=3}^b\partial _i$,
which also has length  less than $2\wh{L}$ by
similar reasoning as in the case of $\a_1$;  again after
possibly re-indexing, we can assume that the end point of
$\a_2$ which does not lie in $\partial _1\cup \partial _2$
lies in $\partial_3$. Continuing inductively, we obtain a
collection of arcs $\{\a_1,\a_2,\ldots \a_{b-1}\}$ in
$\Omega_1$, each with length less than $2\wh{L}$ and
the set
\[
\mathcal{C}:=\partial \Omega_1\cup \a_1\cup \ldots\cup \a_{b-1}
\]
is path connected. Note that if $b=1$, then
$\mathcal{C}=\partial \Omega _1=\partial _1$.

For any pair of points $p',q'\in \mathcal{C}$, the intrinsic
distance $d_{\mathcal{C}}(p',q')$ measured in $\mathcal{C}$
can be realized as the length of an embedded piecewise smooth
arc in $\mathcal{C}$ consisting  of arcs alternating between
arcs in components of $\partial \Omega_1$ and arcs in
$\a_1\cup \ldots \cup \a_{b-1}$. In particular,
\begin{equation}
\label{eq:519}
d_{\Omega_1}(p',q')\leq d_{\mathcal{C}}(p',q')\leq
2(b-1)\wh{L}+\textstyle{\frac{1}{2}\mbox{\rm Length}(\partial \Omega_1).}
\end{equation}

Let $p,q$ be points in $\Omega_1$. Let $p',q'\in \partial \Omega_1$ be
the end points of respective length-minimizing
geodesics in $\Omega _1$ joining $p$ and $q$ to $\partial
\Omega_1$. Applying~\eqref{eq:lemma5.50} to $p$ and $q$
together with the estimate in~\eqref{eq:519}, we have
\[
d_{\Omega_1}(p,q)\leq d_{\Omega_1}(p,\partial \Omega_1)
+d_{\Omega_1}(q,\partial \Omega_1)+d_{\mathcal{C}}(p',q')
\leq 2b\wh{L} +\textstyle{\frac{1}{2}\mbox{\rm Length}(\partial \Omega_1),}
\]
which proves Claim~\ref{claim4.4}.
\end{proof}

\begin{claim}
\label{claim4.5}
Inequality (\ref{eq:lemma5.51}) holds.
\end{claim}
\begin{proof}
Since (\ref{eq:lemma5.51}) is invariant under
re-scaling, it suffices to prove it for $R=1$.
By Claim~\ref{claim4.4}, we have that
(\ref{eq:lemma5.51}) will follow by proving that
\begin{equation}
b\leq 3I(f)+2B(\Sigma)-1.
\label{4.11}
\end{equation}

Recall that $f|_{\Omega_1}$ is
transverse to $\partial \B(1)$ and that
$\partial \Omega _1=\{ \partial _1,\ldots ,
\partial _b\} $. Each $\partial _i$ is a simple
closed curve in $\Sigma$, and $\partial _i$ admits
a small tubular neighborhood $U_i$ in $\Sigma$
which is topologically an annulus. 

Assume for the moment that $\S$ is orientable, and
we will prove that $b\leq g(\Sigma)+e$.
Let $\Delta=\{ \Delta_1,\ldots, \Delta_k\}$ denote
the set of components of $\S\setminus
\Int(\Omega_1)$ and since each of these components
has at least one end, then $k\leq e$.
Given $i\in \{ 1,\ldots ,k\}$, let $A_i$ denote the set of components of
$\partial \Delta_i$ with one of the components
arbitrarily removed; in particular the number of
components in $A:=\cup_{i=1}^k A_i$ is $b-k$.
Note that for each component $\be\in A$, there is a simple closed curve
$\g_{\be}$ in $\S$ that intersects $A$ transversely in a single
point of $\be$, where $\g_{\be}$ consists of an arc in $\Omega_1$ together
with  an arc in the component $\Delta_j\in \Delta$ that has $\be$ in its boundary.
It follows that the collection of simple closed curves $A$ does not separate $\S$ and so,
by the definition of genus, the number of elements in $A$, which is $b-k$, is less than
or equal to $g(\S)$.
Since $k\leq e$, then $b\leq g(\Sigma)+e$, which proves the desired inequality
when $\Sigma$ is orientable.

When $\S$ is non-orientable, then $\Sigma$
is the connected sum of $g(\wt{\S})+1$ projective planes
punctured in $e$ points, where $\wt{\S}$ is the oriented
cover of $\S$, and a similar
argument just carried out in the orientable case shows that $b\leq g(\wt{\S})+e+1$.

According to the hypothesis stated for
inequality~(\ref{eq:lemma5.51}), $f$ is assumed
not to be an embedded  plane. Thus the total
spinning $S$ of $f$ satisfies $S\geq 2$.
If $S=2$, then the extrinsic monotonicity formula
for minimal surfaces implies that either $f$ has
one end with multiplicity 2 (in this case
$f(\Sigma)$ is a plane, $B(\S)=b=1$ and $I(f)=0$,
so~\eqref{4.11} is an equality in this case), or $f$ is injective and has two ends.
In this last case, $f(\Sigma)$ is a catenoid by
Schoen~\cite{sc1}, $b\leq 2$, $B(\S)=0$ and $I(f)=1$,
which implies that~\eqref{4.11} holds in this
case.

If $S\geq 3$ and $\Sigma $ is orientable, then
\[
\begin{array}{ccll}
b&\leq &g(\Sigma) +e&
\\
&\leq &2g(\Sigma)+e & \mbox{(because $g(\S)\geq 0$)}
\\
&\leq &3I(f)-2S-e+2B(\S)+5&   \mbox{(by \eqref{eq:CMindex1})}
\\
&\leq &3I(f)+2B(\S)-2& \mbox{(because $S\geq 3$ and $e\geq 1$),}
\end{array}
\]
hence~\eqref{4.11} holds. Finally, if
$S\geq 3$ and $\Sigma $ is non-orientable, then
\[
\begin{array}{ccll}
b&\leq &g(\wt{\Sigma}) +e+1&
	\\
	&\leq &3I(f)-2S-e+2B(\S)+5&   \mbox{(by \eqref{eq:CMindex1})}
	\\
	&\leq &3I(f)+2B(\S)-2& \mbox{(because $S\geq 3$ and $e\geq 1$),}
\end{array}
\]
hence~\eqref{4.11} again holds.
Therefore, inequality \eqref{4.11} holds in every
case, and as observed above,
this suffices to finish the proof of Claim~\ref{claim4.5}.
\end{proof}	

With the auxiliary property~\eqref{eq:lemma5.51} at hand, we next
prove item 2 of Proposition~\ref{propos5.5}. The first statement for $f$ injective with
image a plane is obvious. Assume $f$ is not in this case and we will prove	\eqref{eq:lemma5.53}
for $R=1$.

First suppose  that $f\colon (\S,p_0) \to (\rth,\vec{0})$ is	injective with image a
catenoid $C$. After a possible rotation of $C$ fixing the origin, we can assume that the
$(x_1,x_3)$-plane $P$ is a plane of symmetry of $C$ and the axis of $C$  is parallel to the
$x_3$-axis.	As we observed previously, for estimating distances between pairs of points in
$\Omega_1$, we may assume that the boundary sphere $\partial \B(1)$	is transverse to $C$.
Then $C\cap P \cap \ov{\B}(1)$ contains a component arc $\G$ with non-vanishing curvature
passing through the origin. By  convexity, $\G$ has length less than the length of the boundary
circle of the disk $P\cap\ov{\B}(1)$, and so, length$(\G)<2\pi$. As the axis of $C$  is parallel to the
$x_3$-axis, we deduce that $\G$ can be  parameterized by its third coordinate as $\G=\{
(x_1(t),0,t) \ | \ t\in [a,b]\} $ for some $-1\leq a<0<b\leq 1$. 
Let $C(1)=C\cap\{(x_1,x_2,x_3)\mid a \leq x_3\leq b\}$; clearly
$\Omega_1\subset C(1)$ and $\Omega_1\cap \partial C(1)=\{(x_1(a),0,a),(x_1(b),0,b))$.
Similar comparison estimates
also prove that each horizontal disk $\{ x_3=t\} \cap \ov{\B}(1) $ with $t\in [a,b]$ intersects
$\Omega_1$ in a connected component $\Lambda(t)$ passing through $(x(t),0,t)\in \G$, and
$\Lambda(t)$ is invariant under reflection across $P$. $\Lambda (t)$ is either a
horizontal circle of radius less than 1, a circular arc of length less than $2\pi$ or just
the point $(x_1(t),0,t)$ when $t=a$ or $t=b$. In particular, for any pair of points $p,q\in \Omega_1$ there exists
a piecewise smooth path in $\Omega _1$ joining $p$ and $q$, which consists of a pair of horizontal
circular arcs that join $p$ and $q$ to $\G$ together with an arc in $\G$ joining the end points of
these two horizontal arcs.  It follows that the distance $d_{\Omega_1}(p,q) <4\pi$. Direct
substitution of $I=1$ and $B=0$ in the
RHS of~\eqref{eq:lemma5.53} shows that
the inequality \eqref{eq:lemma5.53} holds in this
case that $f\colon (\S,p_0) \to (\rth,\vec{0})$ is
injective with $f(\S)=C$. 
	
If $S=2$, then the arguments in the fifth
paragraph of the proof of Claim~\ref{claim4.5}
show that either $f$ is injective with $f(\S)$
being a catenoid (hence \eqref{eq:lemma5.53} holds
by the last paragraph), or else $f(\S)$ is a plane
passing
through the origin with $B(\S)=1$; in this last
case the intrinsic distance between any
two points of $\Omega_1$ is less than or equal to $4$,
and so, \eqref{eq:lemma5.53} is also seen to hold.
	
It remains to show that \eqref{eq:lemma5.53} holds
if $S\geq 3$. Assume now that $S\geq 3$.
We proved in the sixth paragraph of
the proof of Claim~\ref{claim4.5} that
if $S\geq3$, then $b\leq 3I(f)+2B(\S)-2$.  Plugging this estimate of $b$ into the inequality in
Claim~\ref{claim4.4} and using the scale invariance of this inequality,
we get the following estimate for all points $p,q\in \Omega_R$
and for all $R>0$:	
\begin{equation}
\label{improved}
d_{\Omega_{R}}(p,q)\leq 2(3I(f)+2B(\S)-2)\wh{L}R
+\textstyle{\frac{1}{2}\mbox{\rm Length}(\partial \Omega_{R}).}
\end{equation}

By the extrinsic monotonicity
formula, $\pi R^2<\mbox{Area}[f^{-1}(\ov{\B}(R))]\leq \pi S R^2$
for each $R>0$, where the strict inequality holds since $f(\S)$ is assumed not to be
injective with image a plane passing through the
origin. Taking $R=1$ in the first of these inequalities and
$R=2$ in the second one, we deduce that
\begin{equation}
	\label{4.12a}
\mbox{Area}[f^{-1}(\ov{\B}(2)-\B(1))]<4\pi S-\pi .
\end{equation}
By the co-area formula,
\begin{equation}
\label{prop6.3a}
\min_{r\in [1,2]} \mbox{Length}[f^{-1}(\partial
\B(r))] \leq \mbox{Area}[f^{-1}(\ov{\B}(2)-\B(1))].
\end{equation}
Let $\rho \in [1,2]$ be
such that $ \mbox{Length}[f^{-1}(\partial
\B(\rho))]$ equals the minimum in the
LHS of~\eqref{prop6.3a}.
Given $p,q\in \Omega_1$,
\[
\begin{array}{ccll}
d_{\Omega_{\rho}}(p,q)&\leq &2(3I+2B-2)\wh{L}\rho
+\frac{1}{2}\mbox{\rm Length}(\partial \Omega_{\rho})& \mbox{(by \eqref{improved})}
	\\
&\leq & 2(3I+2B-2)\wh{L}\rho
+\frac{1}{2}\mbox{Length}[f^{-1}(\partial \B(\rho))]&
\mbox{(because $\partial \Omega_{\rho}\subset f^{-1}(\partial \B(\rho))$)}
\\
&< & 2(3I+2B-2)\wh{L}\rho
+2\pi S -\frac{\pi}{2}& \mbox{(by \eqref{4.12a} and \eqref{prop6.3a})}
\\
&\leq&
4(3I+2B-2)\wh{L} +\pi (3I-2e+2B+5)-\frac{\pi}{2}
\quad&  \mbox{(by \eqref{eq:lemma5.5A} and $\rho \leq 2$)}
\\
&\leq&
4(3I+2B-2)\wh{L} +\pi (3I+2B+3)-\frac{\pi}{2}
& \mbox{(because $e\geq 1$)}
\end{array}
\]
Since $\rho \leq 2$ and $3I+2B=2\wh{L}^2-3$, then
$d_{\Omega_{2}}(p,q)\leq d_{\Omega_{\rho}}(p,q)
<8\wh{L}^3+2\pi \wh{L}^2-20\wh{L} -\frac{\pi}{2}$,
which proves \eqref{eq:lemma5.53} holds. This
completes the proof of Proposition~\ref{propos5.5}.
\end{proof}

\center{William H. Meeks, III at  profmeeks@gmail.com\\
Mathematics Department, University of Massachusetts, Amherst, MA 01003}
\center{Joaqu\'\i n P\'{e}rez at jperez@ugr.es\\
Department of Geometry and Topology and Institute of Mathematics
(IMAG), University of Granada, 18071, Granada, Spain
\bibliographystyle{plain}
\bibliography{bill}
\end{document}